\documentclass[11pt,reqno]{amsart}

\usepackage{bbm}
\usepackage{mathrsfs}
\usepackage{stmaryrd}
\usepackage{amscd,latexsym,amsthm,amsfonts,amssymb,amsmath,amsxtra}

\usepackage[hyperfootnotes=false, colorlinks, citecolor=RoyalBlue, linkcolor=Bittersweet, urlcolor=link]{hyperref}
\usepackage{color}
\usepackage[usenames,dvipsnames]{xcolor}

\usepackage{graphicx}
\usepackage[all]{xy}

\usepackage{setspace}

\usepackage{tikz}
\usetikzlibrary{calc, intersections}
\usepackage{tikz-cd}

\newcommand{\sfW}{{\mathsf{W}}}\newcommand{\sfV}{{\mathsf{V}}}
\newcommand{\sfG}{{\mathsf{G}}}\newcommand{\sfH}{{\mathsf{H}}}
\newcommand{\sfL}{{\mathsf{L}}}
\newcommand{\sfk}{{\mathsf{k}}}\newcommand{\sfS}{{\mathsf{S}}}

\newcommand{\sL}{{\mathscr{L}}}

\newcommand{\FH}{{\mathbf{H}}} \newcommand{\FG}{{\mathbf{G}}}

\newcommand{\FS}{{\mathbf{S}}}
\newcommand{\FW}{{\mathbf{W}}}\newcommand{\FV}{{\mathbf{V}}}
\newcommand{\FJ}{{\mathbf{J}}}

\newcommand{\BC}{{\mathbb {C}}} 
 \newcommand{\BF}{{\mathbb {F}}}

\newcommand{\BQ}{{\mathbb {Q}}} \newcommand{\BR}{{\mathbb {R}}}

 \newcommand{\BZ}{{\mathbb {Z}}}

 \newcommand{\CB}{{\mathcal {B}}}

\newcommand{\fg}{{\mathfrak{g}}} \newcommand{\fh}{{\mathfrak{h}}}

\newcommand{\fo}{{\mathfrak{o}}} \newcommand{\fp}{{\mathfrak{p}}}

 \newcommand{\fz}{{\mathfrak{z}}}

\newcommand{\Ad}{{\mathrm{Ad}}}

\newcommand{\Aut}{{\mathrm{Aut}}}

\newcommand{\disc}{{\mathrm{disc}}}

\newcommand{\End}{{\mathrm{End}}}

\newcommand{\Gal}{{\mathrm{Gal}}} 
\newcommand{\GL}{{\mathrm{GL}}}
\newcommand{\Hom}{{\mathrm{Hom}}}

\newcommand{\Ind}{{\mathrm{Ind}}} \newcommand{\ind}{{\mathrm{ind}}}

\newcommand{\Lie}{{\mathrm{Lie}}}

\newcommand{\Mp}{{\mathrm{Mp}}}

\newcommand{\Nm}{{\mathrm{Nm}}}
\renewcommand{\O}{{\mathrm{O}}}

\newcommand{\red}{\mathrm{red}}

\newcommand{\SO}{{\mathrm{SO}}}\newcommand{\Sp}{{\mathrm{Sp}}}

\newcommand{\Tr}{{\mathrm{Tr}}}
\newcommand{\R}{{\mathrm{R}}}
\newcommand{\U}{{\mathrm{U}}}
\newcommand{\val}{{\mathrm{val}}}

\newcommand{\pair}[1]{\langle {#1} \rangle}

\newcommand{\incl}{\hookrightarrow}
\newcommand{\sk}{\medskip}

\newcommand{\ra}{\rightarrow} 

\newcommand{\bs}{\backslash}

\newcommand{\s}{\sk\noindent}

\theoremstyle{plain}
\newtheorem{thm}{Theorem}[section] \newtheorem{cor}[thm]{Corollary}
\newtheorem{lem}[thm]{Lemma}  \newtheorem{prop}[thm]{Proposition}

\theoremstyle{definition}
 
\newtheorem{rem}[thm]{Remark}

\numberwithin{equation}{section}

\title{Theta lifts and distinction for regular supercuspidal representations}
\author{Chong Zhang}
\begin{document}
\date{}
\maketitle

\begin{abstract}
This article has a twofold purpose. First, by recent works of Kaletha and Loke--Ma, we give an explicit description of the local theta correspondence between regular supercuspidal representations in the equal rank symplectic-orthogonal case. Second, based on this description, we show that the local theta correspondence preserves distinction with respect to unramified Galois involutions.
\end{abstract}

\tableofcontents

\section{Introduction}
\subsection*{Overview} Theta lifts are concrete realizations of the Langlands functoriality for classical groups or metaplectic groups. In recent years, there are several celebrated developments in this direction, e.g. \cite{lst11,gi14,sz15,gi16,gt16}, just name a few among them. For supercuspidal representations, the cornerstone of admissible representations of $p$-adic reductive groups, a fundamental result of Kudla says that the first occurrence of nonzero theta lifts of supercuspidal representations are also supercuspidal representations. Recently Loke and Ma \cite{lm} gave a description of the local theta correspondence between tame supercuspidal representations in terms of the supercuspidal data. This type of supercuspidal representations was first constructed by Yu \cite{yu01} and developed further by others, e.g. \cite{kim07,hm08}, etc.

In his recent work \cite{kal}, Kaletha found much simpler data, called tame regular elliptic pairs, to parameterize regular supercuspidal representations which are a subclass of tame supercuspidal representations. This new remarkable construction is a generalization of Howe's classical result \cite{how77}, and has many applications in the landscape of the Langlands program. For example, Kaletha constructed the $L$-packets of regular supercuspidal representations (see {\em loc. cit.}), which generalizes previous works \cite{ree08,dr09,dr10,ka14,ry14,kal15}. Hence it is natural to ask for a description of the local theta correspondence between regular supercuspidal representations in terms of tame regular elliptic pairs, which will in turn reflect the relation between the $L$-parameters under theta lifts. We attempt to answer this question in the equal rank symplectic-orthogonal case.

Another application of Kaletha's construction was presented in our previous work \cite{zha2}, which is concerned with the distinction problem. The general theory of distinguished tame supercuspidal representations in terms of supercuspidal data has been studied by Hakim and Murnaghan \cite{hm08}, whose prior or subsequent works treated various typical examples in the case of general linear groups (cf. \cite{hm02a,hm02b,hj12,hak13}). Recently, Hakim \cite{haka,hakb} provided a new appraoch to the construction of tame supercuspidal representations and its application to the distinction problem. On the other hand, it is natural to study the relation between the distinction problem and the Langlands functoriality, especially in the spirit of Sakellaridis and Venkatesh's proposal \cite{sv}, or of Prasad's conjecture \cite{pra} for Galois involutions. In this paper, we will give a simple criterion to detect distinguished regular supercuspidal representations of a general connected reductive group for an unramified Galois involution. Together with our result on theta lifts, this criterion enables us to show that theta lifts preserve distinction.

\subsection*{Main results} Let $E$ be a finite extension field of $\BQ_p$ with $p$ sufficiently large, $\FW$ a $2n$-dimensional symplectic space over $E$ and $\FG=\Sp(\FW)$. We fix an additive character $\psi$ of $E$ of conductor $\fp_E$, where $\fp_E$ is the maximal ideal of the integer ring $\fo_E$ of $E$.

Let $\pi=\pi_{(\FS,\mu,j)}$ be a regular supercuspidal representation of $\FG(E)$, where $\FS$ is a torus over $E$, $\mu$ is a character of $\FS(E)$ and $j:\FS\incl \FG$ is an $E$-embedding such that $(j\FS,\mu)$ is a tame regular elliptic pair of $\FG$. We call $(\FS,\mu,j)$ a regular elliptic triple for short. Our first result (Theorem \ref{thm. theta}) says that up to equivalence there exists a unique $2n$-dimensional quadratic space $\FV$ over $E$ such that the theta lift $\pi'=\theta_{\FV,\FW,\psi}\left(\pi\right)$ is nonzero. Moreover we show that $\pi'$ is a regular supercuspidal representation of $\O(\FV)(E)$, and can be described explicitly in terms of regular elliptic triples. Loke--Ma's description of the local theta correspondence in terms of supercuspidal data is modulo the knowledge about theta lifts over finite fields. However, in our situation, things become easier and clearer, since due to \cite{sri} and \cite{amr96} we have a complete understanding about theta lifts for regular Deligne-Lusztig cuspidal representations.

Next, we turn to the distinction problem. Now we suppose that $F$ is a subfield of $E$ such that $[E:F]=2$ and $E/F$ is unramified.
We also assume that $\FW$ has an $F$-structure coming from $W$, i.e. $\FW$ is obtained from a $2n$-dimensional symplectic space $W$ over $F$ by extension of scalars. Thus $G=\Sp(W)$ is a subgroup of $\FG$, which is fixed by the nontrivial unramified Galois involution $\sigma\in\Gal(E/F)$. Recall that we say $\pi$ is $G(F)$-distinguished if $\Hom_{G(F)}(\pi,{\bf1})$ is nonzero. If $\pi$ is a distinguished regular supercuspidal representation, we will show that the quadratic space $\FV$ in the above paragraph has an $F$-structure coming from $V$ such that the theta lift of $\pi$ to $\O(\FV)$ is $\O(V)(F)$-distinguished (see Theorem \ref{thm. theta and distinction}). The proof of Theorem \ref{thm. theta and distinction} relies
on the results about the distinction problem for general involutions developed in our prior work \cite{zha2}, and also a refinement for unramified Galois involutions given in \S \ref{subsec. criterion}.

It would be of some interests to extend the results in this paper to the dual pair $(\Mp(\FW),\O(\FV))$ where $\Mp(\FW)$ is the metaplectic group with $\dim\FW=2n$ and $\O(\FV)$ is an odd orthogonal group with $\dim\FV=2n+1$. We can also consider the dual pair  $(\Sp(\FW),\O(\FV))$ with $\dim\FW=2n$ and $\dim\FV=2n+2$.

\subsection*{Organization of the paper} The assumptions on the residue characteristic $p$, and necessary notation and convention are given in the rest of this section. In Section \ref{sec. prelim} we recall the background knowledge of supercuspidal data, block decompositions of supercuspidal data for classical groups, and also regular supercuspidal representations. We study the local theta correspondence between regular supercuspidal representations in Section \ref{sec. theta lifts}, and its relation with the distinction problem in Section \ref{sec. distinction}.

\subsection*{Assumptions} Throughout this article, $F$ is a local field which is a finite extension of $\BQ_p$. We assume that $p$ is sufficiently large such that it satisfies both of the assumptions in \cite[\S 1]{zha2} and the hypotheses in \cite[\S 3.4]{kim07}. We refer the reader to \cite[\S 2.1]{kal}, \cite[\S 3.4]{kim07}, \cite[\S 3.1]{lm}, \cite[\S 1]{zha2} for detailed explanations on the roles that these assumptions play in various theories surrounding supercuspidal representations.

\subsection*{Notation and convention} For a $p$-adic field $F$, we denote by $\fo_F$ its ring of integers, by $\fp_F$ the maximal ideal of $\fo_F$, by $\sfk_F$ the residue field, and by $\val_F$ the normalized valuation on $F$ such that $\val_F(F^\times)=\BZ$. For an additive character $\psi$ of $F$ with conductor $\fp_F$, we use $\bar{\psi}$ to denote the additive character on $\sfk_F$ induced by $\psi$. For a quadratic field extension $E$ of $F$, we denote $$E^-=\{x\in E\ |\ \Tr_{E/F}(x)=0\},\ \textrm{and}\ E^1=\{x\in E^\times\ |\ \Nm_{E/F}(x)=1\}.$$

For a reductive group $G$ over $F$, we denote by ${^\circ G}$ its connected component of the identity, and $Z(G)$ the center of $G$. We use the corresponding lowercase gothic letter $\fg$ to denote its Lie algebra. For $\Gamma\in\fg$, let $Z_G(\Gamma)$ denote the stabilizer of $\Gamma$ in $G$.
If $E$ is a finite extension of $F$, we use $\R_{E/F}G$ to denote the Weil restriction of $G$, which is a reductive group over $F$ whose $F$-rational points are $G(E)$.

We denote by $\CB(G,F)$ and $\CB^\red(G,F)$ the extended and reduced Bruhat-Tits buildings of $G(F)$ respectively. For $x\in\CB^\red(G,F)$, we denote by $G(F)_x$ the stabilizer of $x$ in $G(F)$, by $G(F)_{x,0}$ the connected parahoric subgroup attached to $x$, by $G(F)_{x,0+}$ the pro-unipotent radical of $G(F)_{x,0}$, and by $^\circ\sfG_x$ the corresponding connected reductive quotient group over $\sfk_F$. Moreover, if $G$ is a classical group, we denote by $\sfG_x$ the reductive group over $\sfk_F$ such that $\sfG_x(\sfk_F)=G(F)_x/G(F)_{x,0+}$.

As introduced in \cite{mp94}, we denote by $G(F)_{x,r}$ the Moy-Prasad filtration subgroups of $G(F)_{x,0}$ for $r\in\BR_{\geq0}$, and by $\fg(F)_{x,r}$ the filtration lattices of $\fg(F)$ for $r\in\BR$. We write $G(F)_{x,r+}:=\bigcup\limits_{s>r}G(F)_{x,s}$ and $\fg(F)_{x,r+}:=\bigcup\limits_{s>r}\fg(F)_{x,s}$. For $s\geq r$, we write $$G(F)_{x,r:s}:=G(F)_{x,r}/G(F)_{x,s},\quad G(F)_{x,r:s+}:=G(F)_{x,r}/G(F)_{x,s+},$$ $$\fg(F)_{x,r:s}:=\fg(F)_{x,r}/\fg(F)_{x,s},\quad \fg(F)_{x,r:s+}:=\fg(F)_{x,r}/\fg(F)_{x,s+}.$$
For $\Gamma\in\fg(F)$, we denote by $\val(\Gamma)$ the depth of $\Gamma$, that is, 
$$\val(\Gamma)=\sup\limits_{x\in\CB^\red(G,F)}\{r\ |\ x\in\fg_{x,r}\bs\fg_{x,r+} \}.$$

When $G$ is a classical group, there exists a correspondence between the set of self-dual lattice functions and $\CB(G,F)$, or a description of the Moy-Prasad filtration in terms of lattice functions. We refer to \cite[\S 4]{lm16} for a detailed summary.

\subsection*{Acknowledgements} This work was partially supported by NSFC 11501033, Fundamental Research Funds for the Central Universities No. 14380018, and Zheng Gang Scholars Program. The author would like to thank Jia-Jun Ma for kindly answering several questions related to his paper \cite{lm} with Hung Yean Loke, and thank Wen-Wei Li for helpful comments on a preliminary version of this article. He also thanks the anonymous referee for the careful reading and helpful suggestions.

\section{Preliminaries}\label{sec. prelim}

\subsection{Supercuspidal data}
Let $G$ be a reductive group over $F$. We fix an additive character $\psi$ of $F$ of conductor $\fp_F$. Throughout this paper, a {\em supercuspidal $G$-datum} has two meanings. We will abuse the notion when there is no confusion.

First,  a {\em supercuspidal $G$-datum} $\Psi=\left(\vec{G},x,\rho,\vec{\phi}\right)$ means a {\em generic cuspidal $G$-datum} in the sense of \cite{yu01}. To be precise, the datum $\Psi$ statisfies
\begin{enumerate}
\item $\vec{G}$ is a \emph{tamely ramified twisted Levi sequence}
$\vec{G}=(G^0,...,G^d)$ in $G$ such that $Z(G^0)/Z(G)$ is
anisotropic.
\item $x\in\CB^\red(G^0,F)$ is a vertex.
\item $\rho$ is an irreducible representation of $K^0:=G^0(F)_x$ such
that $\rho|_{G^0(F)_{x,0+}}$ is {\bf1}-isotypic and the compactly
induced representation $\pi_{-1}=\ind_{K^0}^{G^0(F)}(\rho)$ is
irreducible.
\item $\vec{\phi}=(\phi_0,...,\phi_d)$ with $\phi_i$ a quasicharacter of $G^i(F)$ for each $0\leq i\leq d$ and being $G^{i+1}$-{\em generic} (cf. \cite[Definition 3.9]{hm08}) for $i\neq d$. We
require that: if $d=0$ then $\phi_0$ is of depth $r_0\geq0$; if
$d>0$ and $\phi_d$ is nontrivial then $\phi_i$ is of depth $r_i$ for
$i=0,...,d$ and $0<r_0<r_1<\cdots<r_{d-1}<r_d$; if $d>0$ and
$\phi_d$ is trivial then $\phi_i$ is of depth $r_i$ for
$i=0,...,d-1$ and $0<r_0<r_1<\cdots<r_{d-1}$. We will call
$\vec{r}=(r_0,...,r_d)$ the {\em depth} of $\vec{\phi}$ for short, and denote
$\phi^\circ:=\prod\limits_{i=0}^d\phi_i|_{G^0(F)}$.
\end{enumerate}

When $G$ is a classical group, things can be simplified as in \cite[\S 3]{lm}. We also adopt the notion therein: a {\em supercuspidal $G$-datum} $\Sigma=(x,\Gamma,\phi,\rho)$ satisfies
\begin{enumerate}
\item $\Gamma$ is a \emph{tamely ramified semisimple} element of $\fg(F)$ and admits a \emph{good factorization} $\Gamma=\sum\limits_{-1\leq i\leq d}\Gamma_i$ (cf. \cite[\S 3.2]{lm}).
\item Set $G^0=Z_G(\Gamma)$. Then $Z\left({^\circ G^0}\right)$ is anisotropic.
    \item $x\in\CB^\red(G^0,F)$ is a vertex.
\item $\rho$ is an irreducible cuspidal representation of the finite group $\sfG^0_x(\sfk_F)$.
\item $\phi:G^0(F)_x\ra\BC^\times$ is a character such that $\phi|_{G^0(F)_{x,0+}}=\psi_\Gamma|_{G^0(F)_{x,0+}}$ (see \cite[\S 3, (3.2)]{lm} for the definition of $\psi_\Gamma$).
\end{enumerate}

We remark that in the above definition we always require ``$\Gamma_{-1}=0$''. See Remarks 2 after \cite[Definition 3.3]{lm} on this point. Given a supercuspidal $G$-datum $\Psi=\left(\vec{G},x,\rho,\vec{\phi}\right)$ such that $\phi_i$ is represented by $\Gamma_i$, we can obtain a supercuspidal $G$-datum $\Sigma=(x,\Gamma,\phi,\rho)$ by setting $\Gamma=\sum\limits_{0\leq i\leq d}\Gamma_i$ and $\phi=\phi^\circ|_{G^0(F)_x}$. Conversely, given a supercuspidal $G$-datum $\Sigma=(x,\Gamma,\phi,\rho)$, there exists a supercuspidal $G$-datum $\Psi=\left(\vec{G},x,\rho,\vec{\phi}\right)$ such that $\phi_i$ is represented by $\Gamma_i$ and $\phi=\phi^\circ|_{G^0(F)_x}$.

When $G$ is connected, Yu obtains an irreducible supercuspidal representation $\pi_\Psi$ or $\pi_\Sigma$ of $G(F)$ from a supercuspidal $G$-datum $\Psi$ or $\Sigma$.  When $G$ is an orthogonal group, as argued in \cite[\S 3.4]{lm}, we can also get an irreducible supercuspidal representation $\pi_\Sigma$ from a supercuspidal $G$-datum $\Sigma$.

\subsection{Block decompositions}
Now we record the block decomposition of a supercuspidal $G$-datum $\Sigma=(x,\Gamma,\phi,\rho)$, which is referred to \cite[\S 4]{lm}. We restrict ourselves to the case that $(V,\pair{\cdot,\cdot}_V)$ is a symplectic or quadratic space over $F$ and $G$ is the corresponding isometry group.

Viewing $\Gamma\in\fg(F)$ as an element of $\End_F(V)$, the algebra $F[\Gamma]$ is isomorphic to a product $\prod\limits_{j\in J}F_j$ of tamely ramified finite extensions $F_j$ of $F$. Since $V$ is an $F[\Gamma]$-module, we have a decomposition $V =\bigoplus\limits_{j\in J}V_j$ where $V_j$ is a subspace and $F_j$ acts faithfully on it. Let $\Gamma^j$ be the $F_j$-component of $\Gamma$ in $\prod\limits_{j\in J}F_j$. Set
$$\{{^br}>\cdots>{^1r}>{^0r}=0\}=\{-\val(\Gamma^j)\geq0\ |\ j\in J\}$$ as (4.1) of {\em loc. cit.}, and put $${^\ell \Gamma}=\sum\limits_{\val(\Gamma^j)=-{^\ell r}}\Gamma^j,\quad {^\ell V}=\bigoplus\limits_{\val(\Gamma^j)=-{^\ell r}}V_j$$ for each $0\leq\ell\leq b$. Note that we have required ``$\Gamma_{-1}=0$'' in the definition of $\Sigma$. Therefore we have ${^0\Gamma}=0$ and ${^0V}=\ker(\Gamma)$. The restriction of $\pair{\cdot,\cdot}_V$ to each ${^\ell V}$ is non-degenerate, whose isometry group is denoted by ${^\ell G}$.

According to Definition 4.4 of {\em loc. cit.}, a {\em depth-zero single block} of $G$ is a supercuspidal $G$-datum of the form $(x,0,\bf{1},\rho)$, and a {\em positive depth single block} of $G$ is a supercuspidal $G$-datum $\Sigma=(x,\Gamma,\phi,\rho)$ such that $\Gamma^j$ has the same negative valuation for all $j\in J$. By Proposition 4.5 of {\em loc. cit.}, there exists a {\em block decomposition} $\bigoplus\limits_{0\leq\ell\leq b}{^\ell\Sigma}=\left({^\ell x},{^\ell\Gamma},{^\ell\phi},{^\ell\rho}\right)$ of $\Sigma=(x,\Gamma,\phi,\rho)$ such that $^0\Sigma$ is a depth-zero single block of $^0G$ and $^\ell\Sigma$ is a positive depth single block of $^\ell G$ for each $1\leq\ell\leq b$. See Proposition 4.5 of {\em loc. cit.} for more details on the block decompositions. Conversely, given single blocks satisfying certain natural conditions, we can obtain a supercuspidal datum by direct sum (cf. Lemma 4.7 of {\em loc. cit.}).

\subsection{Regular supercuspidal representations}

In \cite{kal}, Kaletha defined a subclass of supercuspidal $G$-data, called regular regular supercuspidal data. A supercuspidal $G$-datum $\Psi$ is called {\em regular} if $\rho|_{G^0(F)_{x,0}}$ contains the inflation to $G^0(F)_{x,0}$ of an irreducible cuspidal representation $\kappa$ of $^\circ\sfG^0_x(\sfk_F)$ such that $\kappa$ is a Deligne-Lusztig cuspidal representation $\pm R_\sfS^{^\circ\sfG^0_x}(\lambda)$ attached to an elliptic maximal torus $\sfS$ of $^\circ\sfG^0_x$ and a character $\lambda$ of $\sfS(\sfk_F)$ in general position. An irreducible supercuspidal representation $\pi$ of $G(F)$ is called {\em regular} if it is equivalent to $\pi_\Psi$ for some regular supercuspidal $G$-datum $\Psi$.

In {\em loc. cit.}, Kaletha found more convenient data than supercuspidal data to parameterize regular supercuspidal representations. Recall that a {\em tame regular elliptic pair $(S,\mu)$} of $G$ contains a tame elliptic maximal torus $S$ of $G$ and a character $\mu$ of $S(F)$ satisfying the conditions in Definition 3.6.5 of {\em loc. cit}. From a tame regular elliptic pair $(S,\mu)$, a {\em Howe factorization} of $(S,\mu)$ (cf. \S 3.7 of {\em loc. cit.}) provides a regular supercuspidal $G$-datum $\Psi=(\vec{G},x,\rho_{(S,\mu_0)},\vec{\phi})$ (also called a Howe factorization) where
\begin{enumerate}
\item $S$ is a maximally unramified elliptic maximal torus of $G^0$,
\item $\mu_0$ is a regular depth-zero character of $S(F)$ with respect to $G^0$,
\item $\rho_{(S,\mu_0)}$ is defined and constructed in \S 3.4.3 and \S 3.4.4 of {\em loc. cit.},
\item $\vec{\phi}$ satisfies $\mu=\mu_0\cdot\phi^\circ|_{S(F)}$.
    \end{enumerate}
Let $\pi_{(S,\mu)}$ denote $\pi_\Psi$, which is a regular supercuspidal representation and whose equivalence class does not depend on the choice of Howe factorizations. It is known that all regular supercuspidal representations are of the forms $\pi_{(S,\mu)}$, and $\pi_{(S_1,\mu_1)}$ and $\pi_{(S_2,\mu_2)}$ are equivalent if and only if the pairs $(S_1,\mu_1)$ and $(S_2,\mu_2)$ are $G(F)$-conjugate.

For latter use, $(S,\mu,j)$ is called a {\em regular elliptic triple} if $S$ is an $F$-torus, $\mu$ is a character of $S(F)$, and $j:S\ra G$ is an $F$-embedding such that $(jS,\mu\circ j^{-1})$ is a tame regular elliptic pair of $G$. If there is no confusion, we will write $\mu$ instead of $\mu\circ j^{-1}$ for short. The meaning of the \emph{Howe factorizations} of $(S,\mu,j)$ is obvious. We will denote by $\pi_{(S,\mu,j)}$ the regular supercuspidal representation $\pi_{(jS,\mu)}$.

When $G$ is a classical group, a regular elliptic triple $(S,\mu,j)$ gives rise to a Howe factorization $\left(\vec{G},x,\rho_{(S,\mu_0)},\vec{\phi}\right)$  and thus a supercuspidal $G$-datum $\Sigma=\left(x,\Gamma,\phi,\rho_{(S,\mu_0)}\right)$. We also call $\Sigma$ a Howe factorization of $(S,\mu,j)$.

\subsection{Regular elliptic triples of classical groups}
When $G$ is a classical group, the conjugacy classes of embeddings of maximal tori have a nice parametrization, which is well known (cf. \cite[IV.2]{ss70}, \cite{wal01} or \cite[\S3.1]{li}). As before, we restrict ourselves to symplectic and orthogonal groups, and only consider elliptic maximal tori.
Let $(V,\pair{\cdot,\cdot}_V)$ be a $2n$-dimensional symplectic or quadratic space over $F$ and $G$ the corresponding isometry group.

Consider the datum $$(L,L^\circ,c)=\prod_{i\in I}(L_i,L_i^\circ,c_i),$$ where for each $i\in I$
\begin{enumerate}
\item $L_i^\circ$ is a field extension of $F$ of degree $m_i$,;
\item $L_i$ is a quadratic field extension of $L_i^\circ$;
\item $c_i\in L_i^\times$;
\item $\sum\limits_{i\in I}m_i=n$.
\end{enumerate}
Two data $(L,L^\circ,c)$ and $(K,K^\circ,d)$ are said to be equivalent if there exists an $F$-isomorphism $\varphi:L\ra K$ such that $F(L^\circ)=K^\circ$ and $\varphi(c)d^{-1}\in\Nm_{K/K^\circ}(K^\times)$.

If $c\in L^-:=\prod\limits_{i\in I}L_i^-$, then $(L,\pair{\cdot,\cdot}_c)$ is a $2n$-dimensional symplectic space over $F$, where $\pair{x,y}_c:=\Tr_{L/F}(cx\bar{y})$ and $y\mapsto\bar{y}$ is the involution with respect to the quadratic extension $L/L^\circ$. On the other hand, if $c\in L^\circ$, $(L,\pair{\cdot,\cdot}_c)$ is a $2n$-dimensional quadratic space over $F$, where $\pair{\cdot,\cdot}_c$ is defined in the same way. In any case, the norm-one subgroup $L^1:=\prod\limits_{i\in I}L^1_i$ with respect to $L/L^\circ$ preserves $(L,\pair{\cdot,\cdot}_c)$ under multiplication. For convenience,  $L^1$ will also denote the norm-one $F$-torus whose $F$-rational points are $L^1$. Thus we obtain an embedding of an elliptic maximal torus $$j:S=L^1\incl\U(L),$$ where $\U(L)$ is the isometry group of $(L,\pair{\cdot,\cdot}_c)$.

When $V$ is a symplectic space, the equivalence classes of the data $(L,L^\circ,c)$ with $c\in L^-$ such that $(L,\pair{\cdot,\cdot}_c)\cong (V,\pair{\cdot,\cdot}_V)$ (automatically for symplectic spaces) are bijective with the $G(F)$-conjugacy classes of embeddings of elliptic maximal $F$-tori $j:S=L^1\incl\U(L)\cong G$.

When $V$ is an orthogonal space, the equivalence classes of the data $(L,L^\circ,c)$ with $c\in L^\circ$ such that $(L,\pair{\cdot,\cdot}_c)\cong (V,\pair{\cdot,\cdot}_V)$ are bijective with the $G(F)$-conjugacy classes of embeddings of elliptic maximal $F$-tori $j:S=L^1\incl\U(L)\cong G$. Note that $j$ is actually an embedding into $\SO(V)$.

Therefore the $G(F)$-conjugacy classes of regular elliptic triples $(S,\mu,j)$ are parameterized by the equivalence classes of the data $$(L,L^\circ,c,\chi)=\prod_{i\in I}(L_i,L_i^\circ,c_i,\chi_i),$$ where 
\begin{enumerate}
	\item $(L,L^\circ,c)$ is as above; 
	\item $\chi=\bigotimes\limits_{i\in I}\chi_i=\mu$ is a character of $L^1=\prod\limits_{i\in I}L_i^1=S(F)$ such that the corresponding triple $(S,\mu,j)$ satisfies the extra conditions in \cite[Definition 3.6.5]{kal}.
\end{enumerate}

\section{Theta lifts}\label{sec. theta lifts}
In this section we study the theta lifts of regular supercuspidal representations from symplectic groups to even orthogonal groups. We restrict ourselves to the equal rank case, i.e. the symplectic and quadratic spaces are of the same dimension. The results of \cite{sri} on the theta lifts over finite fields and \cite{lm} on the theta lifts between tame supercuspidal representations are crucial to us. Throughout this section, we suppose that $(W,\pair{\cdot,\cdot}_W)$ is a $2n$-dimensional symplectic space over a $p$-adic field $F$, $\psi$ is a fixed additive character of $F$ of conductor $\fp_F$, and $G=\Sp(W)$.

\subsection{Theta lifts over finite fields}

Let $\mathsf{k}$ be a finite field. Let $\sfW$ and $\mathsf{V}$ be a symplectic and a quadratic vector space over $\mathsf{k}$ of dimension $2n$ respectively. Recall that up to equivalence there are two $2n$-dimensional quadratic spaces $\mathsf{V}^+$ and $\mathsf{V}^-$ over $\mathsf{k}$, whose discriminants are trivial and nontrivial in $\mathsf{k}/\mathsf{k}^{\times2}$ respectively. Denote $\mathsf{G}=\Sp(\mathsf{W})$, $\mathsf{H}^\circ=\SO(\mathsf{V})$ and $\mathsf{H}=\O(\mathsf{V})$.

Analogous to the $p$-adic case, the $\mathsf{G}(\mathsf{k})$-conjugacy classes of elliptic maximal tori $\sfS$ of $\mathsf{G}$ are parameterized by the data $(\sfL,\sfL^\circ)=\prod\limits_{1\leq i\leq r}\left(\sfL_i,\sfL_i^\circ\right)$, where $\sfL_i^\circ$ and $\sfL_i$ are field extensions of $\sfk$ of degree $m_i$ and $2m_i$ respectively, $\sfL_i^\circ$ is a subfield of $\sfL_i$, and $\sum\limits_{1\leq i\leq r}m_i=n$. Under this correspondence, we have $$\sfS(\sfk)\cong\sfL^1:=\prod\limits_{1\leq i\leq r}\sfL_i^1.$$

Similarly, the $\sfH(\sfk)$-conjugacy classes of elliptic maximal tori $\sfS$ of $\sfH^\circ$ are parameterized by the data $(\sfL,\sfL^\circ)=\prod\limits_{1\leq i\leq r}\left(\sfL_i,\sfL_i^\circ\right)$ with one more condition: we require that $r$ is even (resp. odd) if $\sfV$ is equivalent to $\sfV^+$ (resp. $\sfV^-$). We also have $\sfS(\sfk)\cong\sfL^1$.

The above parametrization establishes a correspondence between the conjugacy classes of elliptic maximal tori of $\sfG$ and those of $\sfH^\circ$:
$$\begin{aligned}&\left\{\textrm{elliptic maximal torus of $\sfG$}\right\}/\sfG(\sfk)-\textrm{conj.}\\
\stackrel{1:1}{\longleftrightarrow}&\bigsqcup_{\sfV=\sfV^+,\sfV^-}\left\{\textrm{elliptic maximal torus of $\sfH^\circ$}\right\}/\sfH(\sfk)-\textrm{conj.}\end{aligned}$$ which is explicitly given by
$$\left(\sfS\leftrightarrow(\sfL,\sfL^\circ)\right)\ \mapsto\ \left(\sfS\leftrightarrow(\sfL,\sfL^\circ)\right).$$ Let $\Omega(\sfS,\sfG)=N(\sfS,\sfG)/\sfS$, $\Omega(\sfS,\sfH)=N(\sfS,\sfH)/\sfS$ and $\Omega(\sfS,\sfH^\circ)=N(\sfS,\sfH^\circ)/\sfS$ be the absolute Weyl goups. It is well known that $\Omega(\sfS,\sfH^\circ)$ is a subgroup of $\Omega(\sfS,\sfH)$ of index two, and $\Omega(\sfS,\sfH)$ can be naturally identified with $\Omega(\sfS,\sfG)$.   Taking $\sfk$-rational points, we have
$$\Omega(\sfS,\sfG)(\sfk)=\Omega(\sfS,\sfH)(\sfk)$$ inside $\Aut_\sfk(\sfS)$. We refer to \cite[Proposition 3.1.5]{li} for a description of $\Omega(\sfS,\sfG)(\sfk)$ or $\Omega(\sfS,\sfH)(\sfk)$ in terms of $(\sfL,\sfL^\circ)$.

For an elliptic maximal torus $\sfS$ of $\sfG$ and a character $\mu$ of $\sfS(\sfk)$ in general position, let $\pm R_\sfS^\sfG(\mu)$ be the Deligne-Lusztig cuspidal representation of $\sfG(\sfk)$ attached to $(\sfS,\mu)$. Let $\sfV$ be the quadratic space such that $\sfS\subset\O(\sfV)$, and $\sfV'$ the other quadratic space of dimension $2n$. Write $\sfH=\O(\sfV)$ and $\sfH'=\O(\sfV')$. Choose $\varepsilon\in\sfH(\sfk)$ such that $\det(\varepsilon)=-1$. The $\sfH^\circ(\sfk)$-conjugacy classes in the $\sfH(\sfk)$-conjugacy class of $\sfS$ is represented by $\{\sfS,\sfS^\varepsilon\}$. Since the stabilizer of $\mu$ in $\Omega(\sfS,\sfG)(\sfk)$ is trivial, so is the stabilizer of $\mu$ in $\Omega(\sfS,\sfH)(\sfk)$. Hence we have $$\left(\pm R_\sfS^{\sfH^\circ}(\mu)\right)^\varepsilon\cong\pm R_{\sfS^\varepsilon}^{\sfH^\circ}(\mu^\varepsilon)\ncong\pm R_\sfS^{\sfH^\circ}(\mu).$$ Therefore $$\pm R_\sfS^{\sfH}(\mu):=\Ind_{\sfH^\circ(\sfk)}^{\sfH(\sfk)}\left(\pm R_\sfS^{\sfH^\circ}(\mu)\right)$$ is an irreducible representation of $\sfH(\sfk)$, which is well defined for the $\sfH(\sfk)$-conjugacy class $\sfS$ in $\sfH^\circ$. Combining \cite[Corollary 1]{sri} and \cite[Proposition 2.1]{amr96}, we obtain:

\begin{lem}\label{lem. theta over finite field}
The theta lift of $\pm R_\sfS^\sfG(\mu)$ to $\sfH$ is $\pm R_\sfS^\sfH(\mu)$, and the theta lift of $\pm R_\sfS^\sfG(\mu)$ to $\sfH'$ vanishes.
\end{lem}

\begin{rem}\label{rem. theta over finite field}
Since $-1\in\Omega(\sfS,\sfG)(\sfk)$, we have $\pm R_\sfS^\sfG(\mu)=\pm R_\sfS^\sfG(\mu^{-1})$. Therefore we can rewrite the above lemma: the theta lift of $\pm R_\sfS^\sfG(\mu)$ to $\sfH$ is $\pm R_\sfS^\sfH(\mu^{-1})$.
\end{rem}

\subsection{Depth-zero single block}
We first consider theta lifts of regular depth-zero supercuspidal representations. In this case, since $G$ is split, regular depth-zero supercuspidal representations are parameterized by the {\em unramified} regular elliptic triples $(S,\mu,j)$, that is, we further require that $S$ is unramified and $\mu$ is of depth zero. From now on, we fix a uniformizer $\varpi$ of $\fp_F$.

\begin{prop}\label{prop. depth zero}
Let $\pi_{(S,\mu,j)}$ be a regular depth-zero supercuspidal representation of $G(F)$. Then we have the following statements.
\begin{enumerate}
\item Up to equivalence, there exists a unique $2n$-dimensional quadratic space $V$ over $F$ such that the theta lift $\theta_{V,W,\psi}\left(\pi_{(S,\mu,j)}\right)$ is nonzero.
\item The theta lift $\theta_{V,W,\psi}\left(\pi_{(S,\mu,j)}\right)$ is a regular depth-zero supercuspidal representation attached to certain regular elliptic triple $(S,\mu^{-1},j_\theta)$ of $\O(V)$.
\item Suppose that $(S,\mu,j)$ corresponds to $(L,L^\circ,c,\chi)=\prod\limits_{i\in I}(L_i,L_i^\circ,c_i,\chi_i)$. Then
     \begin{itemize}
     \item $(S,\mu^{-1},j_\theta)$ corresponds to $$(L,L^\circ,c_\theta,\chi^{-1}):=(L,L^\circ,c\tau\varpi,\chi^{-1})=\prod\limits_{i\in I}(L_i,L_i^\circ,c_i\tau_i\varpi,\chi^{-1}_i),$$
         \item $(V,\pair{\cdot,\cdot}_V)\cong(L,\pair{\cdot,\cdot}_{c\tau\varpi})$,\end{itemize} where $\tau=\prod\limits_{i\in I}\tau_i$ with $\tau_i\in L_i^-$ and $\val_{L_i}(\tau_i)=0$ for each $i\in I$.
\end{enumerate}
\end{prop}

\begin{proof}
Suppose that $(S,\mu,j)$ corresponds to $(L,L^\circ,c,\chi)=\prod\limits_{i\in I}(L_i,L_i^\circ,c_i,\chi_i)$. We may assume $(W,\pair{\cdot,\cdot}_W)=(L,\pair{\cdot,\cdot}_c)$. Note that all of $L_i$'s and $L_i^\circ$'s are unramified over $F$. Denote by $\sfL_i$ and $\sfL^\circ_i$ the residue fields of $L_i$ and $L_i^\circ$ respectively. Let
$$I_1=\{i\in I \ |\  \val_{L_i}(c_i)\textrm{ is even} \}\ \textrm{and}\ I_2=\{i\in I \ |\  \val_{L_i}(c_i)\textrm{ is odd} \}.$$
We can assume $\val_{L_i}(c_i)=0$ for $i\in I_1$, and $\val_{L_i}(c_i)=1$ for $i\in I_2$. Let $\sL\subset W$ be the unique lattice stable under $S(\fo')$ for each integer ring $\fo'$ of an unramified extensions $F'$ of $F$.  Exciplicitly, we can write $$\sL=\bigoplus_{i\in I}\fo_{L_i}.$$ Let $\tilde{\sL}$ be the dual lattice of $\sL$ under $\pair{\cdot,\cdot}_W$. Then
$$\tilde{\sL}=\left(\bigoplus_{i\in I_1}\fo_{L_i}\right)\oplus\left(\bigoplus_{i\in I_2}\varpi^{-1}\fo_{L_i}\right).$$ Set $\sfW_1=\sL/\varpi\tilde{\sL}$ and $\sfW_2=\tilde{\sL}/\sL$, which are $\sfk_F$-vector spaces equipped with the induced symplectic forms $\pair{\cdot,\cdot}_W$ and $\varpi\pair{\cdot,\cdot}_W$ respectively.  Note that $$\sfW_i=\bigoplus\limits_{j\in I_i}\sfL_j\textrm{ and } \dim_{\sfk_F}(\sfW_i)=2n_i=\sum\limits_{j\in I_i}2m_j\textrm{ for }
i=1,2.$$

Let $x$ be the vertex of $\CB^\red(G,F)$ attached to $jS$. Then $\Sigma=\left(x,0,{\bf1},\rho_{(S,\mu)}\right)$ is a depth-zero single block attached to $(S,\mu,j)$.  We have $\sfG_x=\sfG_1\times\sfG_2$ where $\sfG_i=\Sp(\sfW_i)$ for $i=1,2$. The elliptic maximal torus of $\sfG_x$ attached to $S$ is $\sfS=\sfS_1\times\sfS_2$ where $\sfS_i$ is the elliptic maximal torus of $\Sp(\sfW_i)$ corresponding to $\prod\limits_{j\in I_i}(\sfL_j,\sfL_j^\circ)$ for $i=1,2$. The depth-zero character $\mu$ of $S(F)$ descends to a character $\bar{\mu}=\bar{\mu}_1\times\bar{\mu}_2$ of $\sfS_1(\sfk_F)\times\sfS_2(\sfk_F)$ in general position. Then the Deligne-Lusztig cuspidal representation $\rho=\rho_{(S,\mu)}$ of $\sfG_x(\sfk_F)$ is $\kappa_1\otimes\kappa_2$ where $\kappa_i=\pm R_{\sfS_i}^{\sfG_i}(\bar{\mu}_i)$ for $i=1,2$.

Now let $V$ be a $2n'$-dimensional quadratic space over $F$ such that $$\pi_V:=\theta_{V,W,\psi}\left(\pi_{(S,\mu,j)}\right)$$ is the first occurrence of the nonzero theta lifts in the Witt tower of $V$. By \cite[Theorem A]{kud86} and \cite[Theorem A]{pan02}, $\pi_V$ is an irreducible depth-zero supercuspidal representation. Let $(y,\rho')$ be a depth-zero $K$-type of $\pi_V$ for $H=\O(V)$, where $y$ is a vertex of $\CB(H,F)$ and $\rho'$ is an irreducible cuspidal representation of $\sfH_y(\sfk_F)$. Let $\sL'\subset V$ be a lattice attached to $y$. Set $\sfV_1=\tilde{\sL}'/\sL'$ and $\sfV_2=\sL'/\varpi\tilde{\sL}'$, which are quadratic spaces over $\sfk_F$. Denote $\sfH_i=\O(\sfV_i)$ for $i=1,2$. Then $\sfH_y=\sfH_1\times\sfH_2$ and $\rho'=\kappa'_1\otimes\kappa'_2$, where $\kappa'_i$ is an irreducible cuspidal representation of $\sfH_i(\sfk_F)$ for $i=1,2$. Note that $(\sfG_1,\sfH_1)$ and $(\sfG_2,\sfH_2)$ are two reductive dual pairs over $\sfk_F$. According to \cite[Theorem 5.6]{pan02}, $\rho'$ must be the theta lift of $\rho$, which by definition means that  $\kappa'_i=\theta_{\sfV_i,\sfW_i,\bar{\psi}}(\kappa_i)$ for $i=1,2$. Since $\kappa_i=\pm R_{\sfS_i}^{\sfG_i}(\bar{\mu}_i)$, by \cite[Theorem]{sri} or Lemma \ref{lem. theta over finite field} we see $$\dim\sfV_i\geq\dim\sfW_i,\quad\forall\ i=1,2.$$ On the other hand, since $$\dim\sfV_1+\dim\sfV_2=\dim V=2n'$$ and we require $\dim V=2n$ in the statement of the proposition, it must be $\dim\sfV_i=\dim\sfW_i$ for $i=1,2$. By Lemma \ref{lem. theta over finite field}, for each $i=1,2$, the quadratic space $\sfV_i$ is unique up to equivalence so that there exists an embedding $\sfS_i\incl\sfH_i$. Hence $V$ is unique up to equivalence. Furthermore we have $\kappa'_i=\pm R_{\sfS_i}^{\sfH_i}(\bar{\mu}^{-1}_i)$. Therefore $\pi_V$ is a regular depth-zero supercuspidal representation.

Actually, if we set $$\left(V,\pair{\cdot,\cdot}_V\right)=\left(L,\pair{\cdot,\cdot}_{c\tau\varpi}\right)\ \textrm{and}\ (S,\mu^{-1},j_\theta)=\left(L,L^\circ,c\tau\varpi,\chi^{-1}\right),$$  the subsets $I_1$ and $I_2$ of $I$ introduced before have the following alternative description:
$$I_1=\{i\in I \ |\  \val_{L_i}(c_i\tau_i\varpi)\textrm{ is odd} \}\ \textrm{and}\ I_2=\{i\in I \ |\  \val_{L_i}(c_i\tau_i\varpi)\textrm{ is even} \}.$$
Let $y$ be the vertex attached to $j_\theta(S)$ and $\sL'$ the corresponding lattice. Then, analogous to the symplectic case, we have $$\tilde{\sL}'/\sL'=\bigoplus\limits_{i\in I_1}\sfL_i=\sfV_1,\quad  \sL'/\varpi\tilde{\sL}'=\bigoplus\limits_{i\in I_2}\sfL_i=\sfV_2,$$ and $\sfH_y=\O(\sfV_1)\times\O(\sfV_2)$. The previous arguments immediately imply that $\theta_{V,W,\psi}\left(\pi_{(S,\mu,j)}\right)=\pi_{(S,\mu^{-1},j_\theta)}$.
\end{proof}

\begin{rem}\label{rem. depth zero classification}
We retain the notation in Proposition \ref{prop. depth zero}. Suppose that $(S,\mu,j)$ corresponds to $(L,L^\circ,c,\chi)=\prod\limits_{i\in I}(L_i,L_i^\circ,c_i,\chi_i)$. Let $I_1$ (resp. $I_2$) be the subset of $I$ containing the elements $i$'s such that $\val_{L_i}(c_i)$'s are even (resp. odd). Set $r=|I_1|$ and $s=|I_2|$. We have
\begin{itemize}
\item if $r$ and $s$ is odd, then $\disc(V)=1$ and $\SO(V)$ is split;
\item if $r$ and $s$ are even, then $\disc(V)=1$ and $\SO(V)$ is non-split;
\item if $r$ is even and $s$ is odd, then $\disc(V)\neq1\in\fo_F^\times/\fo_F^{\times 2}$, the Hasse invariant $\varepsilon(V)=1$, and $\SO(V)$ is quasi-split;
    \item if $r$ is odd and $s$ is even, then $\disc(V)\neq1\in\fo_F^\times/\fo_F^{\times 2}$, the Hasse invariant $\varepsilon(V)=-1$, and $\SO(V)$ is quasi-split.

\end{itemize}
\end{rem}

\subsection{Positive depth single block}
Next we consider the theta lifts of regular supercuspidal representations whose Howe factorizations are positive depth single blocks.

Let $\pi=\pi_{(S,\mu,j)}$ be a regular supercuspidal representation of $G(F)$ such that its Howe factorization $\Sigma=(x,\Gamma,\phi,\rho_{(S,\mu_0)})$ is a single block of positive depth $r$. Note that $\det(\Gamma)\neq 0$. Suppose that $(S,\mu,j)$ corresponds to $(L,L^\circ,c,\chi)$. We may assume $(W,\pair{\cdot,\cdot}_W)=(L,\pair{\cdot,\cdot}_c)$, identify $\Lie(jS)$ with $L^-$, and assume $\Gamma\in\Lie(jS)$.

As introduced in \cite[Definition 5.8]{lm}, we set $V=W$ as an $F$-vector space and equip it with the form $\pair{v_1,v_2}_\Gamma=\pair{v_1,\Gamma v_2}_W$. Since $\Gamma\in\fg(F)$, $(V,\pair{\cdot,\cdot}_\Gamma)$ is a quadratic space. Denote $H=\O(V)$, $\fh=\Lie(H)$. We view both of $G$ and $H$ as subgroups of $\GL(V)$, and view both of $\fg$ and $\fh$ as sub-Lie algebras of $\End(V)$. Then $\Gamma$ also lies in $\fh$,  and  $$H^0=Z_H(-\Gamma)=Z_G(\Gamma)=G^0.$$ The vertex $x\in\CB(G^0)$ is associated with a lattice function $\sL_x$ in $W$. Let $\sL'_x$ be the lattice function in $V$, which is defined in Lemma 5.9 of {\em loc. cit.} Let $x'\in\CB(H^0)$ be the vertex attached to $\sL_x'$. Then $G_x=H_{x'}$. See the Remarks after Proposition 5.13 of {\em loc. cit.} for details. Put $$\Sigma':=\left(x',-\Gamma,\phi^{-1},\rho_{(S,\mu_0)}^\vee\right),$$ which is a single block of $H$ of positive depth $r$.

\begin{prop}\label{prop. positive depth}
We have the following statements.
\begin{enumerate}
\item $\theta_{V,W,\psi}(\pi)=\pi_{\Sigma'}$.
\item $\pi_{\Sigma'}$ is a regular supercuspidal representation attached to certain regular elliptic triple $(S,\mu^{-1},j_\theta)$ of $H$.
\item Suppose that $(S,\mu,j)$ corresponds to $(L,L^\circ,c,\chi)=\prod\limits_{i\in I}(L_i,L_i^\circ,c_i,\chi_i)$. Then
     \begin{itemize}
     \item $(S,\mu^{-1},j_\theta)$ corresponds to $$(L,L^\circ,c_\theta,\chi^{-1})=(L,L^\circ,-c\gamma,\chi^{-1})=\prod\limits_{i\in I}(L_i,L_i^\circ,-c_i\gamma_i,\chi^{-1}_i),$$
         \item $(V,\pair{\cdot,\cdot}_\Gamma)\cong(L,\pair{\cdot,\cdot}_{-c\gamma})$,
             \end{itemize} where $\gamma=\prod\limits_{i\in I}\gamma_i\in L^-$ is $\Gamma$ via the identification $\Lie(jS)=L^-$.
\end{enumerate}
\end{prop}

\begin{proof}
The first statement is due to \cite[Proposition 7.1]{lm} or a special case of the Main Theorem of {\em loc. cit.} By definition $\pi_{\Sigma'}$ is regular.
Suppose that $(S,j)$ corresponds to $(L,L^\circ,c)$. Then $(W,\pair{\cdot,\cdot}_W)=(L,\pair{\cdot,\cdot}_{c})$ and $(V,\pair{\cdot,\cdot}_\Gamma)=(L,\pair{\cdot,\cdot}_{-c\gamma})$. Let $j_\theta:S\incl H$ denote the embedding given by the composition
$$S\stackrel{j}{\incl}G^0=H^0\incl H.$$ Then $(S,j_\theta)$ corresponds to $(L,L^\circ,-c\gamma)$, and $x'\in\CB(H^0)$ is the vertex attached to $j_\theta S$. Since $\rho^\vee_{(S,\mu_0)}=\rho_{(S,\mu_0^{-1})}$, $\Sigma'=(x',-\Gamma,\phi^{-1},\rho^\vee_{(S,\mu_0)})$ is a Howe factorization of $(S,\mu^{-1},j_\theta)$. Therefore $\pi_{\Sigma'}=\pi_{(S,\mu^{-1},j_\theta)}$.
\end{proof}

\subsection{General case}
Now we treat the theta lifts of general regular supercuspidal representations.
Let $\pi=\pi_{(S,\mu,j)}$ be a regular supercuspidal representation of $G(F)$ with a Howe factorization $\Sigma=(x,\Gamma,\phi,\rho)$. Let $\Sigma=\bigoplus\limits_{0\leq\ell\leq b}{^\ell\Sigma}$ be the block decomposition of $\Sigma$ into $b$ positive depth blocks $^\ell\Sigma=\left({^\ell x},{^\ell\Gamma},{^\ell\phi},{^\ell\rho}\right)$ and a depth zero block $^0\Sigma=\left({^0 x},0,{\bf 1},{^0\rho}\right)$, and $W=\bigoplus\limits_{0\leq\ell\leq b}{^\ell W}$ the decomposition of the space $W$.

Since $jS$ is a maximally unramified elliptic maximal torus of   $G^0=\prod\limits_{0\leq\ell\leq b}{^\ell G^0}$, then $S=\prod\limits_{0\leq\ell\leq b}{^\ell S}$ so that $j{^\ell S}$ is a maximally unramified elliptic maximal torus of ${^\ell G^0}$ for each $\ell$. Let $\mu=\prod\limits_{0\leq\ell\leq b}{^\ell\mu}$ be the decomposition of $\mu$ with respect to that of $S$. Then $({^\ell S,{^\ell\mu},j})$ is a regular elliptic triple of ${^\ell G}$ for each $\ell$.
Suppose that $(S,\mu,j)$ corresponds to $$\left(L,L^\circ,c,\chi\right)=\prod\limits_{0\leq\ell\leq b}\left({^\ell L},{^\ell L^\circ},{^\ell c},{^\ell\chi}\right),$$ and $\left({^\ell S,{^\ell\mu},j}\right)$ corresponds to $\left({^\ell L},{^\ell L^\circ},{^\ell c},{^\ell\chi}\right)$. According to Propositions \ref{prop. depth zero} and \ref{prop. positive depth}, we obtain a regular elliptic triple $\left({^\ell S,{^\ell\mu^{-1}},j_\theta}\right)$ associated with the data $\left({^\ell L},{^\ell L^\circ},{^\ell c_\theta},{^\ell\chi^{-1}}\right)$ for each $\ell$. Set $$(S,\mu^{-1},j_\theta)=\prod\limits_{0\leq\ell\leq b}\left({^\ell S,{^\ell\mu^{-1}},j_\theta}\right),$$ and $$(L,L^\circ,c_\theta,\chi^{-1})=\prod\limits_{0\leq\ell\leq b}\left({^\ell L},{^\ell L^\circ},{^\ell c},{^\ell\chi^{-1}}\right).$$

\begin{thm}\label{thm. theta}
Up to equivalence, there exists a unique $2n$-dimensional quadratic space $(V,\pair{\cdot,\cdot}_V)$ over $F$ such that $\pi_V:=\theta_{V,W,\psi}(\pi)$ is nonzero. Moreover we have \begin{enumerate} \item $(V,\pair{\cdot,\cdot}_V)\cong(L,\pair{\cdot,\cdot}_{c_\theta})$,  \item $\pi_V=\pi_{(S,\mu^{-1},j_\theta)}$ is a regular supercuspidal representation, \item $(S,\mu^{-1},j_\theta)$ corresponds to $(L,L^\circ,c_\theta,\chi^{-1})$.
\end{enumerate}
\end{thm}

\begin{proof}
It is a direct consequence of Propositions \ref{prop. depth zero} and \ref{prop. positive depth} and the Main Theorem of \cite{lm}.
\end{proof}

\section{Distinction}\label{sec. distinction}
In this section we investigate the relation between the local theta correspondence and distinction with respect to unramified Galois involutions. We will first give a criterion (see Proposition \ref{prop. distinction}) for testing distinction. It is a generalization of our prior result \cite{zha1}, and is also a refinement of \cite{zha2} for unramified Galois involutions. The key point is that we can show that the character ``$\eta_S$'' is trivial in this situation. Combining this criterion with Theorem \ref{thm. theta}, we derive the main result of this section (see Theorem \ref{thm. theta and distinction}).

\subsection{A criterion}\label{subsec. criterion}
In this subsection, we assume that $G$ is a connected reductive group over $F$, and $E$ is an unramified quadratic field extension of $F$.  We fix an element $\iota\in \fo_E^\times$ such that $\iota\in E^-$. Let $\FG$ denote the Weil restriction $\R_{E/F}G$. The nontrivial automorphism $\sigma\in\Gal(E/F)$ induces an involution, still denoted by $\sigma$, on $\FG$.

Let $\pi_{(\FS,\mu,j)}$ be a regular supercuspidal representation of $\FG(F)$. In \cite[Corollary 3.18]{zha2} we showed that $\pi_{(\FS,\mu,j)}$ is $G(F)$-distinguished if and only if $(\FS,\mu,j)$ is $G(F)$-conjugate to a regular elliptic triple $(\FS,\mu,j')$ of $\FG$ such that $j'\FS$ is $\sigma$-stable and $(\mu\circ j'^{-1})|_{(j'\FS)^\sigma(F)}=\eta_{j'\FS}$. See \cite[Definition 3.8]{zha2} or the proof below for the definition of $\eta_{j'\FS}$. We have the following refinement.

\begin{prop}\label{prop. distinction}
The representation $\pi_{(\FS,\mu,j)}$ is $G(F)$-distinguished if and only if $(\FS,\mu,j)$ is $G(F)$-conjugate to a regular elliptic triple $(\FS,\mu,j')$ such that $j'\FS$ is $\sigma$-stable and $(\mu\circ j'^{-1})|_{(j'\FS)^\sigma(F)}={\bf 1}$.
\end{prop}

\begin{proof}
According to \cite[Corollary 3.18]{zha2}, it suffices to assume that $j\FS$ is $\sigma$-stable and show that $\eta_{j\FS}$ is trivial. Since $j\FS$ is $\sigma$-stable, the torus $S:=(j\FS)^\sigma$ is defined over $F$ and $j\FS=\R_{E/F}S$.

Let $\Psi=(\vec{\FG},x,\rho,\vec{\phi})$ be a Howe factorization of $(\FS,\mu,j)$, and $\vec{r}$ the depth of $\vec{\phi}$. We may assume that $\Psi$ is $\sigma$-symmetric, that is, $\sigma(x)=x$, $\sigma(\vec{\FG})=\vec{\FG}$ and $\vec{\phi}\circ\sigma=\vec{\phi}^{-1}$. Then $G^i:=\left(\FG^i\right)^\sigma$ is defined over $F$ and $\FG^i=\R_{E/F}G^i$ for each $i$. We denote by $\fg^i$ the Lie algebra of $G^i$, and $\fz^i$ the center of $\fg^i$. Set $s_i=\frac{r_i}{2}$.

Now we recall the definition of $\eta_{j\FS}$.
For each $0\leq i\leq d-1$, the quotient group $\FW_i:=\FJ^{i+1}/\FJ^{i+1}_+$ is a symplectic $\BF_p$-vector space, where $$\FJ^{i+1}=\FG^i(F)_{x,r_i}\FG^{i+1}(F)_{x,s_i}\quad\textrm{and}\quad
\FJ^{i+1}_+=\FG^i(F)_{x,r_i}\FG^{i+1}(F)_{x,s_i+}.$$ The symplectic structure of $\FW_i$ is attached to the $\FG^{i+1}$-generic character $\phi_i$ of $\FG^i(F)$. The group $j\FS(F)$ acts on $\FW_i$ by conjugate action and preserves the symplectic structure.  The involution $\sigma$ induces a natural linear transformation on $\FW_i$. Let $W_i$ be the $\sigma$-fixed subspace of $\FW_i$, on which the group $S(F)$ acts by conjugate action. Set $$\chi_i(x)=\det\left(\Ad(x)|_{W_i}\right)^{\frac{p-1}{2}},\quad\forall\ x\in S(F),$$ which is a quadratic character of $S(F)$. Then the character $\eta_{j\FS}$ of $S(F)$ is defined to be $\prod\limits_{0\leq i\leq d-1}\chi_i$.

In the proof of \cite[Lemma 3.9]{zha2} we showed that $\FW_i$ is $j\FS(F)$-equivariant isomorphic to $$\fg^{i+1}(E)_{x,s_i:s_i+}/\fg^{i}(E)_{x,s_i:s_i+},$$ and  $W_i$ is $S(F)$-equivariant isomorphic to $$\fg^{i+1}(F)_{x,s_i:s_i+}/\fg^{i}(F)_{x,s_i:s_i+}.$$ Hence $W_i$ is $S(F)$-equivariant isomorphic to $J^{i+1}/J^{i+1}_+$, where $$J^{i+1}=G^i(F)_{x,r_i}G^{i+1}(F)_{x,s_i}\quad\textrm{and}\quad J^{i+1}_+=G^i(F)_{x,r_i}G^{i+1}(F)_{x,s_i+}.$$

The character $\phi_i$ of $\FG^i(F)$ is realized by a $\FG^{i+1}$-generic element $\Gamma_i\in \fz^i(E)_{-r_i}$. According to \cite[Lemma 5.15]{hm08} we may assume $\sigma(\Gamma_i)=-\Gamma_i$.  Since $\iota$ is in $\fo_E^\times$, we see $\iota\Gamma$ still belongs to $\fz^i(E)_{-r_i}$. On the other hand, $\iota$ is also in $E^-$, so we have $\iota\Gamma_i\in\fz^i(F)_{-r_i}$. It is easy to see that $\iota\Gamma$ is $G^{i+1}$-generic.  Thus $\iota\Gamma_i$ gives rise to a symplectic structure of $W_i$, which is preserved under $S(F)$. Therefore we have a homomorphism $S(F)\ra \Sp(W_i)$. Hence $\chi_i$ is trivial for each $i$, and so is $\eta_{j\FS}$.
\end{proof}

The above proposition has the following direct consequence, whose proof is exactly the same as that of \cite[Corollary 3.17]{zha2}.

\begin{cor}\label{cor. distinction I}
If $\pi$ is a $G(F)$-distinguished regular supercuspidal representation of $\FG(F)$, then $\pi^\vee\cong \pi\circ\sigma$.
\end{cor}

In \cite[\S5]{kal} Kaletha defined the notion {\em regular supercuspidal L-parameters} $\varphi$ for $\FG$. He also constructed the $L$-packets $\Pi_\varphi(\FG)$ of $\varphi$ in the framework of rigid inner twists \cite{kal16}, which consist of certain regular supercuspidal representations of $\FG(F)$. In \cite{zha2} we showed that the sets $$\Pi_\varphi^\vee(\FG):=\left\{\pi^\vee\ |\ \pi\in\Pi_\varphi(\FG)\right\}\ \textrm{ and }\ \Pi_\varphi^\sigma(\FG):=\left\{\pi\circ\sigma\ |\ \pi\in\Pi_\varphi(\FG)\right\}$$ are indeed the $L$-packets attached to some regular supercuspidal $L$-parameters (cf. \cite[Propositions 4.14 and 4.18]{zha2}). The above corollary implies the following result directly, which confirms a conjecture of Lapid in this particular case.

\begin{cor}\label{cor. distinction II}
Suppose that $\pi$ is a $G(F)$-distinguished regular supercuspidal representation of $\FG(F)$ and belongs to the L-packet $\Pi_\varphi(\FG)$. Then we have $\Pi^\vee_\varphi(\FG)=\Pi^\sigma_\varphi(\FG)$.
\end{cor}

\subsection{Theta lifts and distinction}\label{subsec. theta and distinction}
As before, let $E/F$ be an unramified quadratic field extension. We fix a uniformizer $\varpi$ of $\fp_E$ such that $\varpi\in E^-$, and also fix $\iota\in\fo_E^\times$ such that $\iota\in E^-$. Let $\psi$ be a fixed additive character of $E$ of conductor $\fp_E$, and $\psi_\iota$ the character of $F$ defined by $\psi_\iota(a)=\psi(\iota^{-1}a)$.

Let $(W,\pair{\cdot,\cdot}_W)$ be a $2n$-dimensional symplectic $F$-vector space. Set $$(\FW,\pair{\cdot,\cdot}_\FW)=(W,\pair{\cdot,\cdot}_W)\otimes E,$$ which is a $2n$-dimensional symplectic $E$-vector space. Let $G=\Sp(W)$ and $\FG=\Sp(\FW)$.

\begin{thm}\label{thm. theta and distinction}
Let $\pi$ be a $G(F)$-distinguished regular supercuspidal representation of $\FG(E)$. Then there exists a $2n$-dimensional quadratic $E$-vector space $(\FV,\pair{\cdot,\cdot}_\FV)$ equipped with an $F$-structure $(V,\pair{\cdot,\cdot}_V)$ such that $\theta_{\FV,\FW,\psi_\iota}(\pi)$ is nonzero and $\O(V)(F)$-distinguished.
\end{thm}

\begin{proof}

Suppose $\pi=\pi_{(\FS,\mu,j)}$. By Proposition \ref{prop. distinction}, we may assume that $j\FS$ is $\sigma$-stable. Therefore $(j\FS)^\sigma$ is an elliptic maximal torus of $G$. Set $S=j^{-1}(j\FS)^\sigma$. We may also assume $\mu|_{S(F)}=1$ according to Proposition \ref{prop. distinction}. Suppose $(S,j)$ corresponds to some datum $(K,K^\circ,c)=\prod\limits_{i\in I}(K_i,K_i^\circ,c_i)$ over $F$. Then $(\FS,\mu,j)$ corresponds to $(L,L^\circ,c,\chi)=\prod\limits_{i\in I}(L_i,L^\circ_i,c_i,\chi)$ where $L_i=K_i\otimes E$ and $L_i^\circ=K_i^\circ\otimes E$. We may assume $(W,\pair{\cdot,\cdot}_W)=(K,\pair{\cdot,\cdot}_c)$ and $(\FW,\pair{\cdot,\cdot}_\FW)=(L,\pair{\cdot,\cdot}_c)$. Note that $\sigma(c)=c$.

Set $(\FV',\pair{\cdot,\cdot}_{\FV'})=(L,\pair{\cdot,\cdot}_{c_\theta})$ as in Theorem \ref{thm. theta}. Then $\pi_{\FV'}=\theta_{\FV',\FW,\psi}(\pi)$ is nonzero and $\pi_{\FV'}=\pi_{(\FS,\mu^{-1},j_\theta)}$. Put $$\left(L,L^\circ,c_\theta\right)=\left(\prod\limits_{i\in I_0}(L_i,L^\circ_i,c_i\tau_i\varpi)\right)\left(\prod\limits_{i\in I_{>0}}(L_i,L^\circ_i,c_i\gamma_i)\right),$$ where the index $I_0$ corresponds to the depth-zero single block and $I_{>0}$ corresponds to positive depth single blocks.

Recall $\tau_i\in L_i^-$. Since $L_i=K_i\otimes E$ and $L_i^\circ=K_i^\circ\otimes E$, we may require $\tau_i\in K_i^-$. According to \cite[Lemma 5.15]{hm08}, we may assume that $\gamma_i$ satisfies $\sigma(\gamma_i)=-\gamma_i$ for all $i\in I_{>0}$. Therefore $\sigma(c_i\tau_i\varpi)=-c_i\tau_i\varpi$ for each $i\in I_0$ and $\sigma(c_i\gamma_i)=-c_i\gamma_i$ for all $i\in I_{>0}$. In summary, we have $\sigma(c_\theta)=-c_\theta$.

Set $(\FV,\pair{\cdot,\cdot}_\FV)=(\FV',\iota\pair{\cdot,\cdot}_{\FV'})$. Then $(\FV,\pair{\cdot,\cdot}_\FV)=\left(L,\pair{\cdot,\cdot}_{\iota c_\theta}\right)$ and $\sigma(\iota c_\theta)=\iota c_\theta$. We may identify $\O(\FV)$ with $\O(\FV')$ canonically. Let $j_\theta^\iota$ be the composition $\FS\stackrel{j_\theta}{\incl}\O(\FV')=\O(\FV)$. Then $(\FS,j^\iota_\theta)$ corresponds to $(L,L^\circ,\iota c_\theta).$ Since $\sigma(\iota c_\theta)=\iota c_\theta$, we have $\iota c_\theta\in K^\circ$. Hence $\FV$ has an $F$-structure. To be precise, $(\FV,\pair{\cdot,\cdot}_\FV)=(V,\pair{\cdot,\cdot}_V)\otimes E$ where $(V,\pair{\cdot,\cdot}_V)=(K,\pair{\cdot,\cdot}_{\iota c_\theta})$.

Write $\FH^\circ=\SO(\FV)$, $\FH=\O(\FV)$, $H^\circ=\SO(V)$, $H=\O(V)$, $\pi_\FV^\circ=\pi^\circ_{(\FS,\mu^{-1},j_\theta^\iota)}$, and $\pi_\FV=\pi_{(\FS,\mu^{-1},j_\theta^\iota)}$. Note that $\pi_\FV=\theta_{\FV,\FW,\psi_\iota}(\pi)$. It is obvious that $j_\theta^\iota\FS$ is $\sigma$-stable, $(j^\iota_\theta\FS)^\sigma=j^\iota_\theta S$, and $(S,j_\theta^\iota)$ corresponds to $(K,K^\circ,\iota c_\theta)$. Since $\mu^{-1}|_{S(F)}=1$, by Proposition \ref{prop. distinction}, we see that $\pi^\circ_\FV$ is $H^\circ(F)$-distinguished. Since $\pi_\FV=\ind_{\FH^\circ(F)}^{\FH(F)}\pi_\FV^\circ$, according to Mackey theory, we have
$$\begin{aligned}\Hom_{H(F)}(\pi_\FV,{\bf1})&=\bigoplus_{h\in\FH^\circ(F)\bs \FH(F)/H(F)}\Hom_{\FH^\circ(F)\cap {^hH(F)}}(\pi_\FV^\circ,{\bf1})\\
&=\Hom_{H^\circ(F)}(\pi_\FV^\circ,{\bf1}).\end{aligned}$$
Therefore $\pi_\FV$ is $H(F)$-distinguished.

\end{proof}

\s{\small Chong Zhang\\
Department of Mathematics, Nanjing University,\\
Nanjing 210093, Jiangsu, P. R. China.\\
E-mail address: \texttt{zhangchong@nju.edu.cn}}


\begin{thebibliography}{XXXX}
\addtocontents{Bibliography}

\bibitem[AG17]{ag17} H. Atobe and W.T. Gan,
{\em Local theta correspondence of tempered
representations and Langlands parameters}, Invent.
Math. {\bf210} (2015), no. 2, 341--415.

\bibitem[AMR96]{amr96} A.-M. Aubert, J. Michel and R. Rouquier, {\em Correspondance de Howe pour les groupes r\'{e}ductifs sur les corps finis}, Duke Math. J. {\bf83} (1996), no. 2, 353--397.

\bibitem[DR09]{dr09} S. DeBacker and M. Reeder, {\em Depth-zero supercuspidal
L-packets and their stability}, Ann. of Math. (2) {\bf169} (2009),
no. 3, 795--901.

\bibitem[DR10]{dr10} S. DeBacker and M. Reeder,
{\em On some generic very cuspidal representations}, Compos. Math.
{\bf146} (2010) no. 4, 1029--1055.

\bibitem[GI14]{gi14} W.T. Gan and A. Ichino, {\em Formal degrees and local theta correspondence}, Invent. Math. {\bf195} (2014) no. 3,
509--672.

\bibitem[GI16]{gi16} W.T. Gan and A. Ichino, {\em The Gross--Prasad conjecture and local theta correspondence}, Invent.
Math. {\bf206} (2016) no. 3, 705--799.

\bibitem[GT16]{gt16} W.T. Gan and S. Takeda, {\em A proof of the Howe duality conjecture}, J. Am. Math. Soc. {\bf29} (2016) no. 2,
473--493.

\bibitem[Hak13]{hak13} J. Hakim, {\em Tame supercuspidal
representations of $\GL_n$ distinguished by orthogonal involutions},
Represent. Theory {\bf17} (2013), 120--175.

\bibitem[Haka]{haka} J. Hakim, {\em Constructing tame supercuspidal
representations}, Represent. Theory {\bf22} (2018), 45--86.

\bibitem[Hakb]{hakb} J. Hakim, {\em Distinguished cuspidal representations over
$p$-adic and finite fields}, preprint, arXiv:1703.08861.

\bibitem[HL12]{hj12} J. Hakim and J. Lansky, {\em Distinguished tame supercuspidal
representations and odd orthogonal periods},  Represent. Theory
{\bf16} (2012), 276--316.

\bibitem[HM02a]{hm02a}  J. Hakim and F. Murnaghan, {\em Tame supercuspidal
representations of $\GL(n)$ distinguished by a unitary group},
Compos. Math. {\bf133} (2002), no. 2, 199--244.

\bibitem[HM02b]{hm02b}  J. Hakim and F. Murnaghan, {\em Two types of distinguished
supercuspidal representations}, Int. Math. Res. Not. IMRN 2002, no.
35, 1857--1889.

\bibitem[HM08]{hm08}  J. Hakim and F. Murnaghan, {\em Distinguished tame
supercuspidal representations}, Int. Math. Res. Pap. IMRP 2008, no.
2, Art. ID rpn005, 166 pp.

\bibitem[How77]{how77} R. Howe, {Tamely ramified supercuspidal representations of
$\GL_n$}, Pacific J. Math. {\bf73} (1977), no. 2, 437--460.


\bibitem[Kal14]{ka14} T. Kaletha, {\em Supercuspidal $L$-packets via
isocrystals}, Amer. J. Math. {\bf136} (2014), no. 1, 203--239.

\bibitem[Kal15]{kal15} T. Kaletha,
{\em Epipelagic $L$-packets and rectifying characters}, Invent.
Math. {\bf202} (2015), no. 1, 1--89.

\bibitem[Kal16]{kal16} T. Kaletha, {\em Rigid inner forms of real and p-adic groups}, Ann.
of Math. (2) {\bf184} (2016), no. 2, 559--632.

\bibitem[Kal]{kal} T. Kaletha, {\em Regular supercuspidal
representations}, arXiv:1602.03144v2, to appear in J. Amer. Math. Soc.

\bibitem[Kim07]{kim07} J.-L. Kim,
{\em Supercuspidal representations: an exhaustion theorem}, J. Amer.
Math. Soc. {\bf20} (2007), no. 2, 273--320.

\bibitem[Kud86]{kud86} S. Kudla, {\em On the local theta correspondence}, Invent. Math. {\bf83} (1986), 229--255.

\bibitem[LST11]{lst11} J.-S. Li, B. Sun and Y. Tian, {\em The multiplicity one conjecture for local theta correspondences},
Invent. Math. {\bf184} (2011), no. 1, 117--124.

\bibitem[Li]{li} W.-W. Li, {\em Stable conjugacy and epipelagic L-packets for Brylinski--Deligne covers of $\Sp(2n)$}, preprint, arXiv:1703.04365.

\bibitem[LM16]{lm16} H. Loke and J. Ma, {\em Local theta correspondences between epipelagic supercuspidal representations},
Math. Z. {\bf283} (2016), no. 1--2, 169--196.

\bibitem[LM]{lm} H. Loke and J. Ma,
{\em Local theta correspondences between supercuspidal representations}, Ann. Sci. \'Ec. Norm. Sup\'er. (4) {51} (2018), no. 4, 927--991.


\bibitem[MP94]{mp94} A. Moy and G. Prasad, {\em Unrefined minimal $K$-types for $p$-adic
groups}, Invent. Math. {\bf 116} (1994), no. 1--3, 393--408.


\bibitem[Pan02]{pan02} S.-Y. Pan, {\em Local theta correspondence of depth zero representations and theta dichotomy}, J. Math. Soc. Japan {\bf 54} (2002), no. 4, 793-845.

\bibitem[Pra]{pra} D. Prasad, {\em A `relative' local Langlands
correspondence}, preprint, arXiv:1512.04347 (2015).

\bibitem[Ree08]{ree08} M. Reeder,
{\em Supercuspidal L-packets of positive depth and twisted Coxeter
elements}, J. Reine Angew. Math. {\bf620} (2008), 1--33.

\bibitem[RY14]{ry14} M. Reeder and J.-K. Yu,
{\em Epipelagic representations and invariant theory}, J. Amer.
Math. Soc. {\bf27} (2014), no. 2, 437--477.

\bibitem[SV]{sv} Y. Sakellaridis and A. Venkatesh,
{\em Periods and harmonic analysis on spherical varieties},
Ast\'erisque {\bf396} (2017), viii+360 pp.

\bibitem[SS70]{ss70} T. A. Springer and R. Steinberg, {\em Conjugacy classes}, Seminar on Algebraic Groups and Related
Finite Groups, Lecture Notes in
Mathematics, Vol. 131. Springer, Berlin, 1970, 167--266.

\bibitem[Sri79]{sri} B. Srinivasan, {\em Weil representations of finite classical groups}, Invent. Math. {\bf51} (1979), 143--153.

\bibitem[SZ15]{sz15} B. Sun and C.-B. Zhu, {\em Conservation relations for local theta correspondence}, J. Amer. Math.
Soc. {\bf28} (2015), no. (4), 939--983.

\bibitem[Wal01]{wal01} J.-L. Waldspurger, {\em Int\'egrales orbitales nilpotentes et endoscopie pour les groupes classiques
non ramifi\'es}, Ast\'erisque {\bf269} (2001), vi+449 pp.


\bibitem[Yu01]{yu01} J.-K. Yu, {\em Construction of tame supercuspidal
representations}, J. Amer. Math. Soc. {\bf14} (2001), no. 3,
579--622.

\bibitem[Zha18]{zha1} C. Zhang, {\em Distinction of regular depth-zero supercuspidal
L-packets}, Int. Math. Res. Not. IMRN 2018, no. 15, 4579--4601.

\bibitem[Zha]{zha2} C. Zhang, {\em Distinguished regular supercuspidal representations}, preprint, arXiv:1702.04897 (2017).
\end{thebibliography}
\end{document}